\documentclass[12pt]{article}
\usepackage{graphicx}
\usepackage{amsmath, amsthm, amssymb}
\usepackage{mathtools}
\usepackage{fullpage}
\usepackage{thmtools}
\usepackage{tikz}
\usepackage{subfig}
\usetikzlibrary{shapes.geometric, shapes.misc, arrows, arrows.meta, decorations.markings, decorations.pathreplacing, calc, positioning}

\title{On generalizing the Van der Waerden theorem to some symmetric functions}

\author{Arie Bialostocki, Vladyslav Oles}
    \date{}

\definecolor{color1}{HTML}{FF0000}
\definecolor{color2}{HTML}{00FFFF}

\newcommand{\edit}[1]{{{\color{black}{#1}}}}

\newtheorem{theorem}{Theorem}
\newtheorem{lemma}{Proposition}

\newcommand{\ZZ}{\mathbb{Z}}
\newcommand{\NN}{\mathbb{N}}
\newcommand{\FF}{\mathcal{F}}

\DeclareMathOperator*{\defeq}{=}

\begin{document}

\maketitle

\begin{abstract}
Let $n,m$ be positive integers and $c \in \ZZ_n$, where $\ZZ_n$ is the ring of integers modulo $n$. \edit{We address the following problem,} partially solved by N. Alon. Does an infinite sequence over $\ZZ_n$ contain $m$ same-length consecutive blocks $B_1, \ldots, B_m$ s.t. $\sum B_j + c \prod B_j = 0$ for every $j=1,\ldots,m$ (where $\sum B$ and $\prod B$ denote, respectively, the sum and the product of the elements in block $B$)? In the case of $c=0$, this problem is equivalent to the Van der Waerden theorem. \edit{We provide an almost complete answer to the above problem, excluding only the case of square-free $n$ and $c=-1$.} After investigating $B \mapsto \sum B + c\prod B$, we provide related examples of generalizing the Van der Waerden theorem to symmetric functions.
\end{abstract}

\edit{
\section{Background}
In 1927, van der Waerden proved a seminal theorem stating that any finite coloring of the integers contains a monochromatic arithmetic progression of arbitrary length \cite{van1927beweis}. Subsequent generalizations established that this phenomenon persists in far broader contexts, extending to the fields of number theory, logic, algebra, analysis, and computer science. Below we list a few well-known generalizations of the Van der Waerden theorem.
\begin{itemize}
    \item Hales and Jewett showed that any finite coloring of a sufficiently high-dimensional combinatorial cube contains monochromatic combinational lines \cite{hales2009regularity}.
    \item Gallai and Witt extended the Van der Waerden theorem by showing that any finite coloring of the integer lattice $\ZZ^d$ contains a monochromatic affine copy of every finite point configuration
    \cite{witt1952kombinatorischer, soifer2010ramsey}.
    \item Rado characterized all linear equations that have monochromatic solutions in every finite coloring \cite{rado1933studien}.
    \item In the density setting, Szemer{\'e}di proved that any subset of the integers with positive upper density must contain arbitrarily long arithmetic progressions \cite{szemeredi1975sets}, a result that answered the Erd{\"o}s–Tur{\'a}n conjecture and can be seen as a density version of van der Waerden’s theorem. Katznelson and Furstenberg provided an additional proof of the Szemer{\'e}di theorem using ergodic theory \cite{furstenberg1991density}.
    \item Erd{\"o}s and Graham later established the canonical Van der Waerden theorem by proving that any sufficiently large finite coloring of $\NN$ yields either a monochromatic or a rainbow arithmetic progression of arbitrary length \cite{erdos1980old}.
    \item Bergelson and Leibman proved a polynomial generalization of the Van der Waerden theorem by replacing a linear polynomial that defines arithmetic progression with a polynomial of arbitrary degree. \cite{bergelson1996polynomial}.
\end{itemize}
}

\edit{In addition, a dispersed collection of other generalizations recently appeared in Chapters 2--7 of \cite{landman2014ramsey}, too varied to enumerate here.}

\section{Introduction}
\edit{This work considers the zero-sum formulation of the Van der Waerden theorem and generalizes it to various symmetric functions.} The problems addressed in our paper lie on the intersection of combinatorics on words and Ramsey theory and stem from two primary sources.
\begin{itemize}
    \item The first one is the classical work of Thue \cite{thue1906uber},\cite{berstel1995axel} who proved the existence of an infinite sequence over a 3-letter alphabet with no identical consecutive blocks. This seminal paper has developed into a broad theory of combinatorics on words. One notable direction in this field involves assuming an algebraic structure on the alphabet, particularly the ring of integers modulo $n$ \cite{au2011van},\cite{cassaigne2014avoiding}.
\item The second source is the zero-sum Ramsey theory on the integers, addressed in Chapter 10 of \cite{landman2014ramsey}. In essence, the colors in the traditional Ramsey theory are replaced with the elements of $\ZZ_n$, and the notion of monochromatic is replaced with the notion of zero-sum. As will be seen below, the problems addressed in our paper relate to the classical Van der Waerden theorem \cite{graham1991ramsey}.
\end{itemize}


Let $n$ be a positive integer and let $\ZZ_n$ be the ring of integers modulo $n$. Consider an arbitrary sequence $A=\{a_k\}_{k=1}^\infty$ over $\ZZ_n$. A \textit{block} of length $l$ consists of $l>1$ consecutive elements from $A$, relabeled and reindexed as $b_k$'s and denoted by $B = (b_{s+1}, \ldots, b_{s+l}) \in \ZZ_n^l$. Furthermore, we say that blocks $B_1, \ldots, B_m$ are \textit{consecutive} if the first element of block $B_{j+1}$ follows the last element of block $B_j$ for every $j = 1,\ldots,m-1$.
Next, consider a family $\FF = \{f^{(l)}\}_{l=2}^\infty$ where each $f^{(l)}:\ZZ_n^l\to \ZZ_n$ is a function in $l$ variables. For a positive integer $m$, we say that $\FF$ is $m$-vanishing if for every sequence $A=\{a_k\}_{k=1}^\infty$ over $\ZZ_n$ there exist integer $l > 1$ and $m$ consecutive blocks $B_1, \ldots, B_m$ each of length $l$ such that $f^{(l)}(B_1) = \ldots = f^{(l)}(B_m) = 0$. Finally, we say that $\FF$ is \textit{vanishing} if it is $m$-vanishing for every positive integer $m$.

For the sake of simplicity, we will denote the sum and the product over $\ZZ_n$ of all the elements in a block $B$ as, respectively, $\sum B$ and $\prod B$. Furthermore, we will use $\ZZ_n$ arithmetic throughout the paper unless stated otherwise.

In this paper, we investigate the vanishing property of the family $$\FF_c \defeq \{(b_1,\ldots,b_l) \mapsto \sum_{i=1}^l b_i + c \prod_{i=1}^l b_i: l=2,3,\ldots\}$$ for every $c \in \ZZ_n$. It has been motivated by the following Theorem \ref{thm:VdW}, which is equivalent to the Van der Waerden theorem \edit{(consider an auxiliary sequence $A' = \{a'_k\}_{k=1}^\infty$ with $a'_k = \sum_{i=1}^ka_i$ and realize that a monochromatic arithmetic progression of length $m$ in $A'$ is equivalent to $m-1$ consecutive zero-sum blocks in $A$, see e.g. Theorem 4 in \cite{au2011van}).}
\begin{theorem}
    \label{thm:VdW}
    Let $n$ be a positive integer. Then the family $$\FF = \{(b_1,\ldots,b_l)\mapsto \sum_{i=1}^l b_i: l=2,3,\ldots\}$$ is vanishing.
\end{theorem}
Considering $\FF_c$ is a particular case of a wider area of investigation of elementary symmetric polynomials appearing in \cite{bialostocki1990zero}.
 As it can be seen from Theorem \ref{thm:|c|>1} below, we only need to investigate $\FF_c$ for $c \in \{1, -1\}$.
\begin{theorem}
    \label{thm:|c|>1}
    If $n > 1$, then $\FF_c$ is not $m$-vanishing for any $c \in \ZZ_n\setminus\{0,1,-1\}$ and $m>0$.
\end{theorem}
\begin{proof}
    Consider sequence $A = -1, 1, -1, 1,\ldots$ over $\ZZ_n$ and an arbitrary block $B$ within it. Trivially, $\sum B \in \{0,1,-1\}$ and $\prod B \in \{1, -1\}$. Then $\sum B + c\prod B = 0$ implies either $0 = (-1)^kc$ or $(-1)^k = c$, and either way we arrive at a contradiction.
\end{proof}
In Sections \ref{c=1} and \ref{c=-1} we consider the cases of $c=1$ and $c=-1$, respectively. In Section \ref{examples} we provide additional examples of generalizing \edit{Theorem \ref{thm:VdW}, i.e. the zero-sum formulation of} the Van der Waerden theorem.
\section{The case of $c = 1$}
\label{c=1}
The main result of this section is Theorem \ref{thm:c=1} below which gives a complete classification of the $m$-vanishing property of $\FF_1$. The proof follows from several propositions provided further below, and Theorems \ref{thm:p^k} and \ref{thm:pq^2} from Section \ref{c=-1}.

\begin{theorem}
    \label{thm:c=1}
    Let $n>1$. 
    \begin{itemize}
        \item[(a)] If $n \notin \{2,3,4,6,8\}$, then $\FF_1$ is not $m$-vanishing for any $m > 0$.
        \item[(b)] If $n \in \{2,3,4,8\}$, then $\FF_1$ is vanishing.
        \item[(c)] If $n = 6$, then $\FF_1$ is 1-vanishing but not $m$-vanishing for any $m > 1$.
    \end{itemize}
\end{theorem}
\begin{proof}
    Part (c) is a result of combining Propositions \ref{lem:n=6,m=1,c=1} and \ref{lem:n=6,m>1}, and part (b) follows from Theorem \ref{thm:p^k} for the case of $n=2$, Proposition \ref{lem:n=3} for the case of $n=3$, and Proposition \ref{lem:n=4,8} for the case of $n \in \{4,8\}$.

    Let $n = p_1\ldots p_k$ be the decomposition of $n$ into prime factors (not necessarily distinct). If $p_i > 3$ for some $i = 1,\ldots,k$, then either by Proposition \ref{lem:p=4k+1} or Proposition \ref{lem:p=4k+3} there exists an infinite sequence $A$ over $\ZZ_{p_i}$ that does not contain any block $B$ satisfying $\sum B + \prod B \equiv 0 \bmod p_i$. Viewing $A$ as a sequence over $\ZZ_n$ then implies that it cannot contain a block $B$ satisfying $\sum B + \prod B \equiv 0 \bmod n$.
    
    To complete the proof of part (a), it remains to consider the case of $n = 2^h3^{k-h} > 8$ for $0 \leq h \leq k$. It follows that $k \geq 2$, and the subcases of $h = 0$ and $h=k$ are handled by Propositions \ref{lem:9u} and \ref{lem:8u}, respectively. Finally, the proof of the subcase of $0 < h < k$, which implies $k \geq 3$, is given by Theorem \ref{thm:pq^2}.
    
\end{proof}

\begin{lemma}
    If $n=3$, then $\FF_1$ is vanishing.
    \label{lem:n=3}
\end{lemma}
\begin{proof}
    We will follow the idea of Noga Alon \edit{(see Theorem 3.5 (a) in \cite{bialostocki1990zero})}. By the Van der Waerden theorem, there exists a positive integer $w$ s.t. every $3^{m+1}$-coloring of $1,2,\ldots,w$ has a monochromatic arithmetic progression of length $m+1$.\\
    Case 1: every block of length $w$ contains a 0. Because the Van der Waerden theorem guarantees the existence of $m$ consecutive zero-sum blocks $B_1,\ldots,B_m$ each of length $l \geq w$, it must be that $0 = \prod B_j = \sum B_j \quad \forall j=1,\ldots,m$.\\
    Case 2: there exists a block of length $w$ with no 0's, $B = (b_1, \ldots, b_w) \in \{1,2\}^{w}$. Consider a $3^{m+1}$-coloring 
    $$\chi(k) = \left(\sum_{i=1}^k b_i, \prod_{i=1}^k b_i, \begin{cases}
        b_{k-1} & \text{if $k > 1$}\\ 0 & \text{otherwise}
    \end{cases},\ldots,\begin{cases}
        b_{k-m+1} & \text{if $k > m-1$}\\ 0 & \text{otherwise}
    \end{cases}\right)\quad \forall k=1,\ldots,w$$ induced by $B$. Let $\chi(s) = \chi(s + l) = \ldots = \chi(s + ml)$ describe its monochromatic arithmetic progression of length $m+1$, so in particular the blocks $B_j = (b_{s+(j-1)l+1}, \ldots, b_{s+jl})$ satisfy $\sum B_j = 0$ and $\prod B_j = 1$ for every $j=1,\ldots,m$. Because shifting the $j$-th block left by $j-1$ for $j=1,\ldots,m$ does not change its elements due to
    \begin{align*}
        b_{s+jl-t} = b_{s+(j-1)l-t}\quad \forall t=1,\ldots,j-1,
    \end{align*}
    the blocks $\overline{B}_j \defeq (b_{s+(j-1)l-j+2}, \ldots, b_{s+jl-j+1})$ each of length $l$ must also each sum to 0 and multiply to 1. Note that this entails $l \geq 3$, and observe that every non-zero residue in $\ZZ_3$ is its own multiplicative inverse. It follows that the consecutive blocks $\overline{B}_j^{-} \defeq (b_{s+(j-1)l-j+2}, \ldots, b_{s+jl-j})$ each of length $l-1$, obtained by removing the right-most element from $\overline{B}_j$, satisfy $\sum \overline{B}_j^{-} = -\prod \overline{B}_j^{-}$ for every $j=1,\ldots,m$.
    
    
 \end{proof}

\begin{lemma}
    If $n\in\{4,8\}$, then $\FF_1$ is vanishing.
    \label{lem:n=4,8}
\end{lemma}
\begin{proof}
    By the Van der Waerden theorem, there exists a positive integer $w$ s.t. every $n^{m+2}$-coloring of $1,2,\ldots,w$ has a monochromatic arithmetic progression of length $m+1$.\\
    Case 1: every block of length $w$ multiplies to 0. Because the Van der Waerden theorem guarantees the existence of $m$ consecutive zero-sum blocks $B_1,\ldots,B_m$ each of length $l \geq w$, it must be that $0 = \prod B_j = \sum B_j \quad \forall j=1,\ldots,m$.\\
    Case 2: there exists a block $B = (b_1, \ldots b_w) \in \{1,\ldots,n-1\}^{w}$ s.t. $\prod B \neq 0$. Define $b'_k = \begin{cases}
        b_k & \text{if $b_k$ is odd}\\ 1 & \text{otherwise}
    \end{cases}$ for $k=1,\ldots,w$ and consider a $n^{m+2}$-coloring 
    $$\chi(k) = \left(\sum_{i=1}^k b_i, \prod_{i=1}^k b_i, \prod_{i=1}^k b'_i, \begin{cases}
        b_{k-1} & \text{if $k > 1$}\\ 0 & \text{otherwise}
    \end{cases},\ldots,\begin{cases}
        b_{k-m+1} & \text{if $k > m-1$}\\ 0 & \text{otherwise}
    \end{cases}\right)$$ $\forall k=1,\ldots,w$ induced by $B$. Let $\chi(s) = \chi(s + l) = \ldots = \chi(s + ml)$ describe its monochromatic arithmetic progression of length $m+1$, so in particular the blocks $B_j = (b_{s+(j-1)l+1}, \ldots, b_{s+jl})$ satisfy $\sum B_j = 0$ for every $j=1,\ldots,m$. Because $\prod_{i=1}^s b_i = \prod_{i=1}^{s+ml} b_i$ and the order of 2 in the prime factorization of $\prod_{i=1}^k b_i$ is non-decreasing in $k$, it must be that $b_k$ is odd for $s < k \leq s+ml$, which implies $\prod B_j = \prod_{k=s+(j-1)l+1}^{s+jl} b'_k = 1 \quad \forall j=1,\ldots,m$. Moreover, because shifting the $j$-th block left by $j-1$ for $j=1,\ldots,m$ does not change its elements due to
    \begin{align*}
        b_{s+jl-t} = b_{s+(j-1)l-t}\quad \forall t=1,\ldots,j-1,
    \end{align*}
    the blocks $\overline{B}_j \defeq (b_{s+(j-1)l-j+2}, \ldots, b_{s+jl-j+1})$ each of length $l$ must also each sum to 0 and multiply to 1. Note that this entails $l \geq 3$, and observe that every odd residue in $\ZZ_n$ is its own multiplicative inverse. It follows that the consecutive blocks $\overline{B}_j^{-} \defeq (b_{s+(j-1)l-j+2}, \ldots, b_{s+jl-j})$ each of length $l-1$, obtained by removing the right-most element from $\overline{B}_j$, satisfy $\sum \overline{B}_j^{-} = -\prod \overline{B}_j^{-}$ for every $j=1,\ldots,m$.
    
    
 \end{proof}

\begin{lemma}
    If $n=6$, then $\FF_1$ is 1-vanishing.
    \label{lem:n=6,m=1,c=1}
\end{lemma}
\begin{proof}
    Assume that some $A$ does not contain a block whose sum and product add to 0, which implies that it contains neither $(1, 5)$ nor $(5, 1)$. By the Van der Waerden theorem, there exists a positive integer $w$ s.t. every $6^5$-coloring of $1,\ldots,w$ has a monochromatic arithmetic progression of length $2$ whose difference is at least 3. 
    By the assumption, a zero-sum block of length at least $w+3$, whose existence is also guaranteed by the Van der Waerden theorem, cannot multiply to 0. Let $B = (b_1, \ldots, b_{w+3})$ satisfy $\prod B \neq 0$, then one of the following two scenarios must hold.\\
    Case 1: $B \in \{1,3,5\}^{w+3}$. Define $b'_k = \begin{cases}
        b_k & \text{if $b_k \neq 3$}\\ 1 & \text{otherwise}
    \end{cases}$ for $k=2,\ldots,w+1$ and consider a $6^4$-coloring 
    $$\chi(k) = \left(\sum_{i=2}^k b_i, \prod_{i=2}^k b'_i, b_k, b_{k+1}\right) \quad \forall k=2,\ldots,w+1$$ induced by $B$. Let $\chi(s) = \chi(s + l)$ describe its monochromatic arithmetic progression of length $2$ for some $l \geq 3$, so in particular $\sum_{k=s+1}^{s+l}b_k = 0$ and $\prod_{k=s+1}^{s+l}b'_k = 1$. This implies $3 \in \{b_{s+1}, \ldots, b_{s+l}\}$ or else $\prod_{k=s+1}^{s+l}b_k = 1$ which violates the assumption due to $\sum_{k=s+2}^{s+l}b_k = -b_{s+1}$, $\prod_{k=s+2}^{s+l}b_k = b_{s+1}$. Therefore, $\prod_{k=s+1}^{s+l}b_k = \prod_{k=s}^{s+l}b_k = \prod_{k=s+1}^{s+l+1}b_k = 3$, which entails $3 \neq b_s = b_{s+l}$ and $3 \neq b_{s+1} = b_{s+l+1}$ or else $(b_s, \ldots, b_{s+l})$ or $(b_{s+1}, \ldots, b_{s+l+1})$ violates the assumption. Because 1 and 5 cannot be neighbors in $A$, it must be that $b_s = b_{s+1} = b_{s+l} = b_{s+l+1}$. But $b_{s+l+2}$ can be neither $b_{s+l+1}$ (or else $(b_s, \ldots, b_{s+l+2})$ violates the assumption) nor 3 (or else $(b_{s+2}, \ldots, b_{s+l+2})$ does), and we arrive at a contradiction.\\
    Case 2: $B \in \{1,2,4,5\}^{w+3}$. Define $b'_k = \begin{cases}
        b_k & \text{if $b_k \notin \{2,4\}$}\\ 1 & \text{otherwise}
    \end{cases}, b''_k = \begin{cases}
        3 & \text{if $b_k = 2$}\\ 0 & \text{otherwise}
    \end{cases}$ for $k=2,\ldots,w+1$ and consider a $6^5$-coloring 
    $$\chi(k) = \left(\sum_{i=2}^k b_i, \prod_{i=2}^k b'_i, \sum_{i=2}^k b''_i, b_k, b_{k+1}\right) \quad \forall k=2,\ldots,w+1$$ induced by $B$. Let $\chi(s) = \chi(s + l)$ describe its monochromatic arithmetic progression of length $2$ for some $l \geq 3$, so in particular $\sum_{k=s+1}^{s+l}b_k = 0$ and $\prod_{k=s+1}^{s+l}b'_k = 1$. Analogously to Case 1, block $(b_{s+1}, \ldots, b_{s+l})$ contains an even residue, and because its count of 2's is even due to $\sum_{k=s+1}^{s+l} b''_k = 0$ it must be that $\prod_{k=s+1}^{s+l} b_k = 4$. It follows that neither $b_{s+1}$ nor $b_{s+l}$ is an even residue, or else removing it would yield a block violating the assumption. Therefore, $b_s = b_{s+1} = b_{s+l} = b_{s+l+1} \notin \{2, 4\}$. But this value can be neither 1 (or else $(b_s, \ldots, b_{s+l+1})$ violates the assumption) nor 5 (or else $(b_{s+2}, \ldots, b_{s+l-1})$ does), and we arrive at a contradiction.
\end{proof}

\begin{lemma}
    If $n=6$ and $m > 1$, then neither $\FF_{-1}$ nor $\FF_1$ is $m$-vanishing.
    \label{lem:n=6,m>1}
\end{lemma}
\begin{proof}
    Consider $A = 1,3,5,3,1,3,5,3,\ldots$, and assume that it contains consecutive blocks $B_1,B_2$ each of length $l$ s.t. $\sum B_j - \prod B_j = 0$ or $\sum B_j + \prod B_j = 0$ for $j=1,2$. Because every block of $A$ contains a 3 it must be that $\prod B_j = 3$ and therefore $\sum B_j = 3$ for $j=1,2$. It follows that $B_1$ and $B_2$ each contain an odd number of 3's and an even number of 1's and 5's combined, so in particular $l=4k+r$ for $r\in\{1,3\}$. Then $4 \nmid 2l$ and therefore $2l$ consecutive elements from $A$ must contain an odd number of 3's, which yields a contradiction.
\end{proof}

\begin{lemma}[Theorem 3.6 in \cite{bialostocki1990zero}]
    \label{lem:p=4k+1}
    If $n$ is a prime satisfying $n \equiv 1 \bmod 4$, then $\FF_1$ is not $m$-vanishing for any $m > 0$.
\end{lemma}

\begin{lemma}
    \label{lem:existence_of_xyr}
    If $n > 3$ is a prime satisfying $n \equiv 3 \bmod 4$, then there exist $x,y \in \ZZ_n\setminus\{0\}$ and $r \in \{2, 3\}$ s.t. $x + ry = 0$ and $xy^r = 1$.
\end{lemma}
\begin{proof}
One of the following scenarios must hold.\\
    Case 1: 4 is a cubic residue. Let $x \in \ZZ_n$ satisfy $x^3 = 4$, then $y = -\frac{x}{2}$ satisfy $x+2y=0$ and $xy^2=\frac{x^3}{4}=1$.\\
    Case 2: 4 is not a cubic residue. Because every $a\in\ZZ_n$ is a cubic residue when $n\equiv2\bmod3$, it must be that $n\equiv1\bmod3$, and therefore every $a \in \ZZ_n$ has either 0 or 3 distinct cubic roots in $\ZZ_n$. In particular, $0 = a^3 - 1 = (a-1)(a^2+a+1)$ has two solutions besides the unity, at least one of them being a root of $a^2 + a + 1 = 0$. The discriminant of this quadratic polynomial is $-3$, and therefore there exists $z \in \ZZ_n$ satisfying $z^2 = -3$. it follows that $x = (3z)^{\frac{n+1}{4}}, y = -\frac{x}{3}$ satisfy $x+3y=0$ and $xy^3=\frac{-(3z)^{n+1}}{27}=1$.
\end{proof}

\begin{lemma}
    \label{lem:p=4k+3}
    If $n > 3$ is a prime satisfying $n \equiv 3 \bmod 4$, then $\FF_1$ is not $m$-vanishing for any $m > 0$.
\end{lemma}
\begin{proof}
    One can verify that the sequences 
    $A=2,3,3,3,3,2,3,3,3,3,\ldots$ and $A=5,3,3,5,3,3,\ldots$ satisfy the statement of the theorem for, respectively, 
    $n=7$ and $n=11$. Therefore, it only remains to consider the case of $n>11$.

    Let $x,y\in\ZZ_n\setminus\{0\}$ and $r\in\{2,3\}$ satisfy $x=-ry$ and $ry^{r+1}=-1$ as in Proposition \ref{lem:existence_of_xyr}, and consider the sequence $A = x,\underbrace{y,\ldots,y}_{\text{$r$ times}},x,\underbrace{y,\ldots,y}_{\text{$r$ times}},\ldots$. Any block $B$ must then satisfy $\sum B = sx+ty$ and $\prod B =x^sy^t$ for some $s \in \{0, 1\}, t \in \{0,\ldots, r\}$ s.t. $s+t < 1+r$. Assume some $B$ satisfies $\sum B = -\prod B$, which is equivalent to $(rs-t)y = (-r)^sy^{s+t}$. It immediately follows that $(s, t) \notin \{(0,0), (0, 1), (1, 0)\}$, and the remaining cases are examined below.\\
    Case 1: $s=0, t=2$. Then $-2y=y^2$, which implies $y=2$, and therefore $-1=ry^{r+1}\in \{32, 48\}$. It follows that $n \mid 33$ or $n \mid 49$.\\    
    Case 2:  $s=0, t=r=3$. Then $-3y=y^3$, which implies $y^2=-3$, and therefore $-1=3y^4=27$. It follows that $n \mid 28$.\\
    Case 3: $s=1,t=1$. Then $(r-1)y=-ry^2$, which implies $y=(1-r)r^{-1}$ and therefore $-1=(1-r)^{r+1}r^{-r} \in \{-2^{-2}, 16\cdot3^{-3}\}$. It follows that $n \mid 5$ or $n \mid 25$.\\
    Case 4: $s=1,t=2,r=3$. Then $y=-3y^3$, which implies $y^2=-3^{-1}$ and therefore $-1=3y^4 = 3^{-1}$. It follows that $n \mid 4$.\\
    Because the prime divisors of $33, 49, 28, 5, 25, 4$ are at most 11, we arrive at a contradiction.
\end{proof}

\begin{lemma}
    \label{lem:8u}
    If $n = 8u$ for $u > 1$, then $\FF_1$ is not $m$-vanishing for any $m>0$.
\end{lemma}
\begin{proof}
    Consider the sequence $A = 3, -3, 3, -3, \ldots$ and assume that some block $B$ of length $l$ satisfies $\sum B + \prod B = 0$.\\
    Case 1: $l$ is even. Then $\sum B = 0$ but $0 \neq \prod B \in \{\pm 3^k: k=1,2,\ldots\}$, which yields a contradiction.\\
    Case 2: $l = 4r+3$ for some integer $r \geq 0$. Then either $\sum B = 3, \prod B = -3^{4r+3}$ or $\sum B = -3, \prod B = 3^{4r+3}$, so $0 = \sum B + \prod B$ entails $0 = 3(3^{4r+2} - 1) = 3^{4r+2}-1 = 2(3^{4r+1} + 3^{4r} + \ldots + 1)$ and therefore $4u \mid 3^{4r+1} + 3^{4r} + \ldots + 1$. Notice that $u>1$ entails $r \neq 0$ due to $4u \nmid 4$. Because $3^{2h}\equiv1\bmod4$ and $3^{2h+1}\equiv3\bmod4$ for any integer $h\geq0$, $3^{4r+1} + 3^{4r} + \ldots + 1 \equiv 3^{4r+1} \equiv 3 \bmod 4$, which yields a contradiction.\\
    Case 3: $l = 4r+1$ for some integer $r > 0$. Then either $\sum B = 3, \prod B = 3^{4r+1}$ or $\sum B = -3, \prod B = -3^{4r+1}$, so $0 = \sum B + \prod B$ entails $0 = 3(3^{4r} + 1) = 3^{4r}+1$ which yields a contradiction because $3^{4r}\equiv1\bmod4$.
\end{proof}

\begin{lemma}
    \label{lem:9u}
    If $n = 9u$ for $u > 0$, then $\FF_1$ is not $m$-vanishing for any $m > 0$.
\end{lemma}
\begin{proof}
    The proof will assume the context of $\ZZ_9$ arithmetic unless stated otherwise. Consider the sequence $A = 7, 4, 4, 7, 4, 4, \ldots$. Any block $B$ in this sequence is comprised of $r + t$ 7's and $2r + s$ 4's for some $r = 0, 1, \ldots$, $t \in \{0, 1\}$, and $s \in \{0, 1, 2\}$ s.t. $t + s < 3$, and therefore $\sum B = 6r + 4s + 7t \equiv s + t \bmod 3$ and $\prod B = 4^{r+s}7^t \equiv 1 \bmod 3$. Assume that some block $B$ satisfies $\sum B + \prod B = 0$, so in particular $\sum B + \prod B \equiv 0 \bmod 3$ which implies $s + t \equiv 2 \bmod 3$.\\
    Case 1: $s = t = 1$. Then $\sum B + \prod B = 6r + 2 + 4^{r+1}7 = 3$, which yields a contradiction.\\
    Case 2: $s = 2, t = 0$. Then Then $\sum B + \prod B = 6r + 8 + 4^{r+2} = 6$, which yields a contradiction.
    Because $A$ does not contain any block $B$ satisfying $\sum B + \prod B = 0$, neither it can contain a block $B$ s.t. $\sum B + \prod B \equiv 0 \bmod n$.
\end{proof}

\section{The case of $c = -1$}
\label{c=-1}
In this section we provide a classification of the $m$-vanishing property of $\FF_{-1}$ for every $n$, excluding only the case where $n$ is square-free.

\begin{theorem}
    If $n=p^k$ for a prime $p$ and $k > 0$, then $\FF_{-1}$ is vanishing.
    \label{thm:p^k}
\end{theorem}
\begin{proof}
By the Van der Waerden theorem, there exists a positive integer $w$ s.t. every $n^{m+2}$-coloring of $1,2,\ldots,w$ has a monochromatic arithmetic progression of length $m+1$.\\
    Case 1: every block of length $w$ in the sequence multiplies to 0. Because the Van der Waerden theorem guarantees the existence of $m$ consecutive zero-sum blocks $B_1,\ldots,B_m$ each of length $l \geq w$, it must be that $0 = \prod B_j = \sum B_j \quad \forall j=1,\ldots,m$.\\
    Case 2: there exists a block $B = (b_1, \ldots b_w) \in \{1,\ldots,n-1\}^{w}$ s.t. $\prod B \neq 0$. Define $b'_k = \begin{cases}
        b_k & \text{if $p \nmid b_k$}\\ 1 & \text{otherwise}
    \end{cases}$ for $k=1,\ldots,w$ and consider a $n^{m+2}$-coloring
    $$\chi(k) = \left(\sum_{i=1}^k b_i, \prod_{i=1}^k b_i, \prod_{i=1}^k b'_i, b_{k+1}, \ldots, b_{k+m-1}\right)\quad \forall k=1,\ldots,w$$ induced by $B$. Let $\chi(s) = \chi(s + l) = \ldots = \chi(s + ml)$ describe its monochromatic arithmetic progression of length $m+1$, so in particular the blocks $B_j = (b_{s+(j-1)l+1}, \ldots, b_{s+jl})$ satisfy $\sum B_j = 0$ for every $j=1,\ldots,m$. Because $\prod_{i=1}^s b_i = \prod_{i=1}^{s+ml} b_i$ and the order of $p$ in the prime factorization of $\prod_{i=1}^k b_i$ is non-decreasing in $k$, it must be that $p \nmid b_k$ for $s < k \leq s+md$ and therefore $\prod B_j = \prod_{i=s+(j-1)l+1}^{s+jl} b'_i = 1 \quad \forall j=1,\ldots,m$. Moreover, because shifting the $j$-th block right by $j-1$ for $j=1,\ldots,m$ does not change its elements due to
    \begin{align*}
        b_{s+(j-1)l+t} = b_{s+jd+t}\quad \forall t=1,\ldots,j-1,
    \end{align*}
    the blocks $\overline{B}_j \defeq (b_{s+(j-1)l+j}, \ldots, b_{s+jl+j-1})$ must also each sum to 0 and multiply to 1. It follows that the consecutive blocks $\overline{B}_j^{+} \defeq (b_{s+(j-1)l+j}, \ldots, b_{s+jl+j})$ each of length $l+1$, obtained via extending $\overline{B}_j$ by one element on the right, satisfy $\sum \overline{B}_j^{+} = \prod \overline{B}_j^{+}$ for every $j=1,\ldots,m$.
\end{proof}

\begin{theorem}
    If $n$ satisfies $pq^2 \mid n$ for distinct primes $p, q$, then neither $\FF_1$ nor $\FF_{-1}$ is $m$-vanishing for any $m>0$.
    \label{thm:pq^2}
\end{theorem}
\begin{proof}
    We repurpose the proof of Theorem 3.5 (b) in \cite{bialostocki1990zero} as follows. Consider the sequence $A = q, -q, q, -q, \ldots$ and an arbitrary block $B$ within it. Trivially, $\sum B \in \{0,q,-q\}$ and $\prod B \in \{\pm q^k: k = 2,3,\ldots\}$. Assume $\sum B + c\prod B = 0$ for some $c \in \{1, -1\}$, which is equivalent to $n \mid \sum B + c\prod B$. It follows that $\sum B \neq 0$, or otherwise $\sum B + c\prod B \in \{q^k, -q^k\}$ for some integer $k>1$, which yields a contradiction because $p \nmid q^k$ implies that $n \nmid q^k$. But $\sum B \in \{q, -q\}$ would entail $\sum B + c\prod B \in \{q + q^k, -q+q^k, q-q^k, -q-q^k\}$ for some integer $k>1$, and therefore $n \mid q(q^{k-1} - 1)$ or $n \mid q(q^{k-1} + 1)$. However, neither $q \nmid q^{k-1} - 1$ nor $q \nmid q^{k-1} + 1$, which contradiction concludes the proof.
\end{proof}

\section{Multivariate generalization of the Van der Waerden theorem}
\label{examples}
The family of sums $\FF = \{(b_1,\ldots,b_l)\mapsto \sum_{i=1}^l b_i: l=2,3,\ldots\}$ can be viewed as an instance of families of \textit{\edit{transformation} sums}, $$\FF_g = \{(b_1,\ldots,b_l)\mapsto \sum_{i=1}^l g(b_i): l=2,3,\ldots\},$$ in which the \textit{transformation} $g:\ZZ_n\to\ZZ_n$ is taken as the identity. 
Another common example of a transformation is $x \mapsto x^r$ for some integer $r > 1$, which yields the family of power sums of degree $r$.

Note that the $m$-vanishing property is naturally extended to families of vector-valued functions $f^{(l)}:\ZZ_n \to \ZZ_n^d$ for $l=2,3,\ldots$ by replacing the scalar zero in $f^{(l)}(B_1) = \ldots = f^{(l)}(B_m) = 0$ with its vector counterpart $\mathbf{0} \in \ZZ^d$, i.e. by requiring all $d$ components of $f^{(l)}$ to simultaneously attain 0 on each of the $m$ consecutive blocks. This allows a further generalization of the families of \edit{transformation} sums by considering multiple transformations at once:
$$\FF_{(g_1,\ldots,g_d)} = \left\{(b_1,\ldots,b_l)\mapsto \left(\sum_{i=1}^l g_1(b_i), \ldots, \sum_{i=1}^l g_d(b_i)\right): l=2,3,\ldots\right\}.$$
\edit{In the following simple result, we show} that any function whose finite-dimensional value is comprised of the \edit{transformation} sums yields a vanishing family\edit{:}
\begin{theorem}
    \label{thm:vector-valued vanishing}
    Let $n$ and $d$ be positive integers and $g_i:\ZZ_n\to\ZZ_n$ for $i=1,\ldots,d$. Then the family $\FF_{(g_1,\ldots,g_d)}$ 
    is vanishing.
\end{theorem}
\begin{proof}
Consider a $n^d$-coloring $$\chi(k) = \left(\sum_{i=1}^l g_1(b_i), \ldots, \sum_{i=1}^l g_d(b_i)\right) \quad \forall k=1,2\ldots$$
induced by $A$. Let $\chi(s) = \chi(s + l) = \ldots = \chi(s + ml)$ describe its monochromatic arithmetic progression of length $m+1$ for some $l \geq 2$. Because the blocks $B_j = (b_{s+(j-1)l+1}, \ldots, b_{s+jl})$ satisfy $\sum_{k=1}^l g_i(b_{s+(j-1)l+k}) = 0$ for every $j=1,\ldots,m$ and $i=1,\ldots,d$, the result immediately follows.
\end{proof}

In particular, choosing any integer $r > 0$ and applying Theorem~\ref{thm:vector-valued vanishing} for $d=r$ and $g_i=x\mapsto x^i$ for $i=1,\ldots,d$ implies that the family of \edit{the combined first $r$ power sums
$$\FF_{\text{powers},r} = \left\{(b_1,\ldots,b_l)\mapsto \left(\sum_{i=1}^l b_i, \sum_{i=1}^l b_i^2, \ldots, \sum_{i=1}^l b_i^r\right): l=2,3,\ldots\right\}$$
is vanishing. Because the elementary symmetric polynomial of degree $r$ can be expressed as a polynomial in the first $r$ power sums with no constant term via the Newton identities, it follows that the family of elementary symmetric polynomials of degree $r$ 
$$\FF_{\text{elem},r} = \left\{(b_1,\ldots,b_l)\mapsto \sum_{1 \leq i_1 < i_2 < \ldots < i_r \leq l} b_{i_1}b_{i_2}\ldots b_{i_r}: l=2,3,\ldots\right\}$$
is also vanishing.} This presents a simpler proof of Theorem 3.4 from \cite{bialostocki1990zero}.

\edit{
\section*{Acknowlegment}
The authors would like to thank the referee for her/his insightful comments, which improve the presentation of this paper.
}

\bibliographystyle{alpha}
\bibliography{references/references.bib}

\end{document}